\documentclass[11 pt,twoside]{article}
\usepackage{amssymb}
\usepackage{picture}
\usepackage{amsmath}
\usepackage{amsthm}
\newtheorem{lemma}{Lemma}[section]
\newtheorem{theorem}{Theorem}[section]

\newtheorem{remark}{Remark}[section]

\newtheorem{prop}{Proposition}[section]
\numberwithin{equation}{section}

\begin{document}

\title{ON THE EXACT INVERSE PROBLEM OF THE CALCULUS OF VARIATIONS}

\author{\footnotesize Veronika Chrastinov\'a*, V\'aclav Tryhuk**}

\maketitle

\begin{abstract}
The article concerns the problem if a~given system of differential equations is identical with the Euler--Lagrange system of an~appropriate variational integral. Elementary approach is applied. The main results involve the determination of the first--order variational integrals related to the second--order Euler--Lagrange systems.
\end{abstract}

\medskip


\bigskip

In the broadest sense, the inverse problem of the calculus of variations concerns the indication of latent extremality principles. The \emph{exact inverse problem} appears as a~very strict subcase: to determine if a~given system of differential equations is \emph{identical} with the Euler--Lagrange system of appropriate variational integral. This problem was investigated for a~long time, however, the results still cannot be regarded as satisfactory.

Our aim is twofold. First, to demonstrate the simplicity of the well--known achievements on the exact inverse problem. If the classical methods of mathematical analysis are applied, the resulting Propositions \ref{prop1.1}--\ref{prop1.4} can be included into elementary textbooks. Second, to propose a reasonable method of determination of the ``most economical'' solutions of the exact inverse problem for the most important case of the second--order Euler--Lagrange systems of partial differential equations.

All literature is actually subject to the strong mechanisms of the jet theory, we may refer to \cite{T1a}--\cite{T9a} for typical examples. The variational integrals are defined on total spaces of jet bundles $J^r(\pi)$ of a~fibered manifolds $\pi: E\rightarrow M$ where $M$ represents the independent variables which are strongly distinguished from the dependent variables on the fibers of~$\pi.$
This approach ensures the rich structure (with Euler--, Dedecker--, Helmholz--, Sonin-- mappings, a~hierarchy of projections, horizontal and vertical forms, connections, the variational sequence, and so like) but also a~certain rigidity of the powerful tools: admissible mappings should preserve the fibration $\pi$ and the order of jets, the singularities in fibers are rejected, multivalued solutions are omitted. As a~result, the true nature of problems to be resolved may be obscured, see Appendix for a~curious example of this unpleasant reality.

We return to the the idea of article \cite{T10a} which is adapted for partial differential equations. Alas, this approach cannot be expressed in terms of the jet theory.

\section{Preface}\label{sec1}
In order to outline the core of actual achievements, let us introduce the \emph{jet coordinates}
\begin{equation}\label{eqn1.1}x_i,\ w^j_I\quad (i=1,\ldots,n;\, j=1,\ldots,m;\, I=i_1\cdots i_r;\, r=0,1,\ldots\,).\end{equation}
They are called \emph{independent variables} $x_1,\ldots,x_n,$ \emph{dependent variables} $w^1,\ldots,w^m$ (empty $I=\emptyset$ with $r=0)$ and \emph{higher--order variables} $w^j_I$ (nonempty $I$ with $r\geq 1)$ which correspond to derivatives
\[\frac{\partial w^j}{\partial x_I}=\frac{\partial^rw^j}{\partial x_{i_1}\cdots\partial x_{i_r}}\quad (I=i_1\cdots i_r;\, i_1,\ldots,i_r=1,\ldots,n)\] in the familiar sense. For a~good reason, we always suppose $i_1\leq\ldots\leq i_r$ from now on.

We shall deal with $C^\infty$--smooth functions, each depending on a~finite number of coordinates (\ref{eqn1.1}). However, in this Preface, the functions may depend on a~parameter $t.$ So the primary independent variables are completed with the additional term $t\, (=x_{n+1})$ and the higher--order variables $w^j_I$ with the additional \emph{variations} $w^j_{It},w^j_{Itt},\ldots\,$ which correspond to derivatives
\[\frac{\partial}{\partial t}\frac{\partial w^j}{\partial x_I}\,(=\frac{\partial}{\partial x_{n+1}}\frac{\partial w^j}{\partial x_I}), \frac{\partial^2}{\partial t^2}\frac{\partial w^j}{\partial x_I}\,(=\frac{\partial^2}{\partial x^2_{n+1}}\frac{\partial w^j}{\partial x_I}), \ldots\,.\]
In fact only the \emph{first--order variations} $w^j_{It}$ are important. Altogether we speak of the \emph{extendend jet coordinates}.

Some functions $w^j=w^j(x_1,\ldots,x_n,t)$ will be substituted into the functions $F=F(x_1,\ldots,x_n,t,\cdot\cdot,w^j_I,w^j_{It},\cdot\cdot)$ under consideration. Let 
\[\mathcal F=\mathcal F(x_1,\ldots,x_n,t)\] \[=F\left(x_1,\ldots,x_n,t,\cdot\cdot,\frac{\partial w^j}{\partial x_I}(x_1,\ldots,x_n,t),\frac{\partial}{\partial t}\frac{\partial w^j}{\partial x_I}(x_1,\ldots,x_n,t),\cdot\cdot\right)\]
be the result of substitution. Then
\[\frac{\partial\mathcal F}{\partial x_i}=\frac{d}{dx_i}F=F_{x_i}+\sum w^j_{Ii}F_{w^j_I}+\sum w^j_{Iit}F_{w^j_{It}}+\cdots\qquad (i=1,\ldots,n),\]
\[\frac{\partial\mathcal F}{\partial t}=\frac{d}{dt}F=F_t+\sum w^j_{It}F_{w^j_I}+\sum w^j_{Itt}F_{w^j_{It}}+\cdots\]
in terms of \emph{total derivatives} $d/dx_i, d/dt.$
The iterations
\[\frac{d}{dx_I}=\frac{d}{dx_{i_1}}\cdots\frac{d}{dx_{i_r}},\ \frac{d}{dt}\frac{d}{dx_I}=\frac{d}{dt}\frac{d}{dx_{i_1}}\cdots\frac{d}{dx_{i_r}}\quad (I=i_1\cdots i_r)\]
make good sense, too. In accordance with the common practice, the presence of various substitutions need not be always explicitly declared since it will be clear from the context.

With this preparation, let $f=f(\cdot\cdot,x_i,w^j_I,\cdot\cdot)$ be a~fixed function of variables (\ref{eqn1.1}), the \emph{Lagrange function}. Repeated use of the rule
\[g\,w^j_{Iit}=-\frac{d}{dx_i}g\cdot w^j_{It}+\frac{d}{dx_i}(g\,w^j_{It})\]
yields the classical identity
\begin{equation}\label{eqn1.2}
\frac{d}{dt}f=\sum f_{w^j_I}w^j_{It}=e[f]+d[f]\quad (e[f]=\sum e^j[f]w^j_t, d[f]=\sum\frac{d}{dx_i}F_i)\end{equation}
where
\begin{equation}\label{eqn1.3} e^j[f]=\sum (-1)^r\frac{d}{dx_I}f_{w^j_I}\quad (j=1,\ldots,m;\, I=i_1\cdots i_r;\, i_1\leq\ldots\leq i_r)\end{equation}
are the \emph{Euler--Lagrange expressions} and $d[f]$ the \emph{divergence component} with (not uniquely determined) coefficients $F_i$ linearly depending on variations $w^j_{It}.$
%%%
\begin{prop}[uniqueness]\label{prop1.1} Let
\begin{equation}\label{eqn1.4}\frac{d}{dt}f=\sum e^jw^j_t+\sum\frac{d}{dx_i}G_i\end{equation}
where $e^j$ are functions of variables $(\ref{eqn1.1})$ while $G_i$ may also depend on variations. Then $e^j[f]=e^j;$ $j=1,\ldots,m;$ and $d[f]=\sum\frac{d}{dx_i}G_i.$\end{prop}
\begin{proof} Assuming (\ref{eqn1.4}), substituting $w^j=w^j(x_1,\ldots,x_n,t)$ and denoting
\[dx=dx_1\wedge\cdots\wedge dx_n,\ dx^i=-(-1)^idx_1\wedge\cdots\wedge dx_{i-1}\wedge dx_{i+1}\wedge\cdots\wedge dx_n,\]
we obtain identity
\[\int_\Omega\frac{d}{dt}fdx=\int_\Omega(e[f]+d[f])dx=\sum\int_\Omega e^jw^j_tdx+\sum\int_{\partial\Omega}G_idx^i\]
by integration over a~domain $\Omega$ in the space of independent variables. Therefore
\[\sum\int_\Omega(e^j-e^j[f])w^j_tdx+\sum\int_{\partial\Omega}(G_i-F_i)dx^i=0\]
identically, for all variations.
The classical argument may be applied: if variations $w^j_t=w^j_t(x_1,\ldots,x_n,t)$ are vanishing near boundary $\partial\Omega,$ the second summand disappears and this implies the desired result. \end{proof}
%%%
\begin{prop}[divergence]\label{prop1.2}
We claim that $e[f]=0$ if and only if \[f=\sum\frac{d}{dx_i}f_i\] where $f_1,\ldots,f_n$ are appropriate functions of variables $(\ref{eqn1.1}).$
\end{prop}
\begin{proof} Assuming $e[f]=0,$ let us insert functions
\begin{equation}\label{eqn1.5} tw^j+(1-t)c^j\quad (j=1,\ldots,m;\,c^j=c^j(x_1,\ldots,x_n))\end{equation}
($c^j$ may be arbitrary but fixed functions) for variables $w^j$ into identity (\ref{eqn1.2}). We obtain
\begin{equation}\label{eqn1.6} f|_{t=1}-f|_{t=0}=\int_0^1\frac{d}{dt}f\,dt=\int_0^1\sum\frac{d}{dx_i}F_i\,dt=\sum\frac{d}{dx_i}\int_0^1F_i\,dt\end{equation}
by subsequent integration where the first term $(t=1)$ is identical with the original function $f,$ the second term $(t=0)$ is a~certain divergence
\[c=c(x_1,\ldots,x_n)=\frac{d}{dx_1}\int c\,dx_1\] and the right--hand integrals are functions of variables (\ref{eqn1.1}).

The converse is easier since trivially
\[\frac{d}{dt}f=\frac{d}{dt}\sum\frac{d}{dx_i}f_i=\sum\frac{d}{dx_i}\frac{d}{dt}f_i\]
and the uniqueness implies $e[f]=0.$ \end{proof}
%%%
\begin{prop}[Helmholz]\label{prop1.3}
Let $e^1,\ldots,e^m$ be functions of variables $(\ref{eqn1.1}).$ Identity
\begin{equation}\label{eqn1.7} e[F]=0\quad (F=\sum e^jw^j_t)\end{equation} in the extended jet space is satisfied if and only if
\begin{equation}\label{eqn1.8} e^j=e^j[f]\quad (j=1,\ldots,m)\end{equation} for appropriate Lagrange function $f$ of variables $(\ref{eqn1.1}).$ \end{prop}

Before passing to the proof, let us discuss identity $(\ref{eqn1.7}).$ In the \emph{extendend jet space}, we have Euler--Lagrange expressions
\[e^j[F]=\sum (-1)^r\frac{d}{dx_K}F_{w^j_K}\qquad (j=1,\ldots,m;\, K=k_1\cdots k_r;\,r=0,1,\ldots\,)\]
where $k_1,\ldots,k_r=1,\ldots,n+1$ (and $k_1\leq\cdots\leq k_{n+1}).$ We prefer notation $x_{n+1}$ for the parameter $t$ at this place. Clearly
\[e^j[F]=\sum F^j_Iw^j_{It}\] where $F^j_I$ are functions only of variables (\ref{eqn1.1}). 	It follows that identity (\ref{eqn1.7}) is equivalent to the \emph{Helmholz conditions}
\begin{equation}\label{eqn1.9} F^j_I=0\quad (\text{all }j\text{ and }I)\end{equation} for the given functions $e^1,\ldots,e^m.$
\begin{proof} Identity (\ref{eqn1.2}) reads
\[F=\frac{d}{dt}f-\sum\frac{d}{dx_i}F_i\quad (F=e[f]=\sum e^j[f]w^j_t),\]
hence $e[F]=0$ if $F$ is regarded as a~Lagrange function in the extended jet space and Proposition~\ref{prop1.1} is applied.

Let us conversely assume (\ref{eqn1.7}), therefore
\begin{equation}\label{eqn1.10} F=\sum\frac{d}{dx_i}G_i+\frac{d}{dt}G\end{equation}
for appropriate functions $G_i$ and $G$ if Proposition~\ref{prop1.2} is applied in extended jet space. Subsequent substitution (\ref{eqn1.5}) and integration yields the identity
\[\int_0^1Fdt=\sum\frac{d}{dx_i}\int_0^1G_i\,dt+G|_{t=1}-G|_{t=0}\]
analogous to (\ref{eqn1.6}). It follows that $G$ (the term $t=1$) depends only on coordinates (\ref{eqn1.1}). Applying uniqueness to (\ref{eqn1.10}), we conclude that $F=e[G].$\end{proof}
%%%
\begin{prop}[Tonti]\label{prop1.4} Let us introduce the Lagrange function
\begin{equation}\label{eqn1.11}\tilde f=\int_0^1e[f]\,dt=\sum\int_0^1e^j[f]\,dt\,(w^j-c^j)\end{equation}
where variables $(\ref{eqn1.5})$ were inserted into the integral. Then $e[\tilde f]=e[f].$
\end{prop}
\begin{proof} Equation (\ref{eqn1.2}) together with substitution (\ref{eqn1.5}) and integration implies
\[f|_{t=1}-f|_{t=0}=\int_0^1e[f]\,dt+\sum\frac{d}{dx_i}\int_0^1F_i\,dt=\tilde f+\sum\frac{d}{dx_i}\int_0^1F_i\,dt.\]
The first summand ($t=1$) is identical with the original function $f,$ the second one ($t=0$) is a~mere function $c=c(x_1,\ldots,x_n)$ of independent variables. It follows that \[\frac{d}{dt}f=\frac{d}{dt}\tilde f+\frac{d}{dt}\sum\frac{d}{dx_i}f_i=\frac{d}{dt}\tilde f+\sum\frac{d}{dx_i}\frac{df_i}{dt}\quad (f_i=\int_0^1F_i\,dt)\] and the proof is done by applying Propositions~\ref{prop1.2} and~\ref{prop1.1}.\end{proof}

We have in principle resolved the \emph{exact inverse problem}: if $e^1=\cdots=e^n=0$ is a~given system of differential equations, the \emph{Helmholz conditions} (\ref{eqn1.9}) are necessary and sufficient for the existence of a~(uncertain at this place) Lagrange function $f$ satisfying (\ref{eqn1.8}). Then the \emph{Tonti integral} (\ref{eqn1.11}) provides the explicit solution of the problem since $e[\tilde f]=e[f].$ Alas, this is a~very ``degenerate'' solution.
In more detail, there are obvious inequalities
\[\text{order }\tilde f\leq \text{ order }e^j[f]\leq 2\text{ order }f\]
where the possible existence of more interesting ``nondegenerate'' Lagrange function $f$ is not yet ensured.
The actual literature to this problem rests on rigid mechanisms of the jet spaces completed with the theory of the Poincar\'e--Cartan forms and variational bicomplex. Our next aim is to propose more flexible and elementary method within the framework of the classical analysis.

%%%
Interrelations of the above Propositions~\ref{prop1.1}--\ref{prop1.4} and the common geometrical approach are not quite clear and
can be \emph{informally} stated as follows. We introduce \emph{vector fields} $Z$ and \emph{contact forms} $\omega^j_I$ such that \[Z=\sum w^j_{It}\frac{\partial}{\partial w^j_I},\ \omega^j_I=dw^j_I-\sum w^j_{Ii}dx_i\,.\]
Let $\Omega$ be the \emph{contact module} of all forms $\sum a^j_I\omega^j_I$ (finite sum) and $\varphi$ the \emph{Poincar\'e--Cartan form} to the Lagrange function $f$ defined by the properties
\[\varphi=fdx+\sum a^j_{Ii}\,\omega^j_I\wedge dx^i,\ d\varphi\cong\sum e^j[f]\omega^j\wedge dx\ (\text{mod }\Omega\wedge\Omega).\]
There are two formulae
\[\mathcal L_Z\varphi\cong\sum f_{w^j_I}w^j_{It}dx,\ \mathcal L_Z\varphi\cong\left(\sum e^j[f]w^j_t+\sum\frac{d}{dx_i}\sum a^j_{Ii}w^j_I\right)dx\quad (\text{mod }\Omega)\]
for the Lie derivative obtained either directly or by the rule $\mathcal L_Z=Z\rfloor d+dZ\rfloor$ and it follows that
\[\frac{d}{dt}f=\sum f_{w^j_I}w^j_{It}=\sum e^j[f]w^j_t+\sum\frac{d}{dx_i}\sum a^j_{Ii}w^j_I=e[f]+d[f].\]
We have obtained the crucial equation (\ref{eqn1.2}). It is however not easy to establish the geometrical sense of this \emph{formal procedure} where the Lagrange function $f,$ the form $d\varphi$ (represented by $e[f]$) and the form $\varphi$ (represented by $d[f]$) are related in one equation.
%%%

For better clarity, let us mention the ``introductory'' Lagrange function $f=f(x,y,y')$ of the classical calculus of variations where the jet coordinates are replaced with the common notation for this moment. The Euler--Lagrange expression
\[e^1[f]=f_y-(f_{y'})'=f_y-f_{y'x}-f_{y'y}y'-f_{y'y'}y''=E(x,y,y',y'')\]
provides the Tonti integral
\[\tilde f=\int_0^1E\left(x,ty+(1-t)c,ty'+(1-t)c',ty''+(1-t)c''\right)dt\,(y-c)\]
where $c=c(x)$ is a~fixed (in principle arbitrary) function. If function $E$ is given in advance, then $\tilde f$ essentially differs from the primary function $f.$ For instance, if we choose elementary
\[f=\frac{e^y}{y'},\ e^1[f]=2\frac{e^y}{y'}\left(1-\frac{y''}{(y')^2}\right)\]
then the Tonti integral
\[\tilde f=2\int_0^1\frac{e^{ty+(1-t)x}}{ty'+(1-t)}\left(1-\frac{ty''}{(ty'+1-t)^2}\right)\,dt\,(y-x)\]
\[=2\int_1^{y'}e^{(\tau-1)\frac{y-x}{y'-1}+x}\left(1-\frac{\tau-1}{\tau^2}\frac{y''}{y'-1}\right)\frac{d\tau}{\tau}\,\frac{y-x}{y'-1}\]
(where $c=c(x)=x$) is a~terrible higher transcendence. We will propose another method here which is as follows. If the function $E$ is given, the second derivative $f_{y'y'}$ of the sought function $f$ (the coefficient of $y''$ in $E$) is well--known, therefore the original Lagrange function $f$ is uniquely determined modulo  linear correction $A(x,y)y'+B(x,y).$ Appropriate choice of this correction ensures the remaining summands $f_y-f_{y'x}-f_{y'y}y'$ of a~given function $E$ and easily provides the ``good'' solution, see below.

The aim of this article is twofold. First, we deal with the problem whether a~second--order system of \emph{partial} differential equations is \emph{identical} with the Euler--Lagrange system of a~\emph{first--order} variational integral. This is a~highly difficult problem and though we propose a~general strategy of the solution, a~reasonable result is achieved only for the case of two dependent variables at most. Second, we aim to provide short and elementary approach to some well--known actual achievements. As a~result, they turn into (almost) trivialities, see also the Appendix. We conclude that the direct elementary methods can be more effective than the sophisticated theories.
%%%
%%%%%%%%%%%%%%%%%%%%%%%%%%%%%%%%%%%%%%%%%%%%%%%%
\section{General strategy}\label{sec2}
The jet variables (\ref{eqn1.1}) are enough from now on. We are going to deal systematically with the first--order Lagrange function $f=f(\cdot\cdot,x_i,w^j,w^j_i.\cdot\cdot).$ Then the Euler--Lagrange expressions are
\[e^j[f]=E^j[f]-\sum E^{jj'}_{ii'}[f]w^{j'}_{ii'}\quad (j=1,\ldots,m)\] where
\[E^j[f]=f_{w^j}-\sum f_{w^j_ix_i}-\sum f_{w^j_iw^{j'}}w^{j'}_i,\ E^{jj'}_{ii'}[f]=\frac{1}{2}\left(f_{w^j_iw^{j'}_{i'}}+ f_{w^j_{i'}w^{j'}_i}\right).\]
The exact inverse problem is as follows. Let
\[F^{jj'}_{ii'}(\cdot\cdot), F^j(\cdot\cdot)\qquad (i,i'=1,\ldots,n;\, j,j'=1,\ldots,m;\, (\cdot\cdot)=(\cdot\cdot,x_i,w^j,w^j_i,\cdot\cdot))\]
be given functions. We ask the question whether the requirements
\begin{equation}\label{eqn2.1} F^{jj'}_{ii'}=E^{jj'}_{ii'}[f],\ F^j=E^j[f]\quad (i,i'=1,\ldots,n;\,j,j'=1,\ldots,m) \end{equation}
are satisfied for an~appropriate first--order Lagrange function $f.$

In order to simplify the notation, we abbreviate
\[f^j_i=f_{w^j_i},\ f^{jj'}_{ii'}=f_{w^j_iw^{j'}_{i'}},\ \ldots, \ f^j_{ix_{i'}}=f_{w^j_ix_{i'}}',\ \ldots\]
from now on. Let us moreover introduce the auxiliary functions \[G^{jj'}_{ii'}=\frac{1}{2}\left(f^{jj'}_{ii'}-f^{jj'}_{i'i}\right)=\frac{1}{2}\left(f_{w^j_iw^{j'}_{i'}}-f_{w^j_{i'}w^{j'}_i}\right).\]
Then
\begin{equation}\label{eqn2.2} f^{jj'}_{ii'}=F^{jj'}_{ii'}+G^{jj'}_{ii'}\quad (i,i'=1,\ldots,n;\, j,j'=1,\ldots,m). \end{equation}
The symmetry properties
\begin{equation}\label{eqn2.3} F^{jj'}_{ii'}=F^{jj'}_{i'i}=F^{j'j}_{i'i}=F^{j'j}_{ii'},\ G^{jj'}_{ii'}=-G^{jj'}_{i'i}=G^{j'j}_{i'i}=-G^{j'j}_{ii'} \end{equation}
are postulated.

We are passing to the topic proper.

\bigskip

The \emph{first requirement} (\ref{eqn2.1}) is equivalent to the Pfaffian system
\[d_1f=\sum f^j_idw^j_i,\ d_1f^j_i =\sum f^{jj'}_{ii'}dw^{j'}_{i'}\quad (f^{jj'}_{ii'}=F^{jj'}_{ii'}+G^{jj'}_{ii'})\]
where differential $d_1$ is applied only to the first--order variables. Due to the symmetry properties (\ref{eqn2.3}), the first Pfaffian equation is always solvable. The second system of the Pfaffian equations is solvable if and only if 
\[f^{jj'j''}_{ii'i''}=\left(f^{jj'}_{ii'}\right)_{w^{j''}_{i''}}=\left(f^{jj''}_{ii''}\right)_{w^{j'}_{i'}}=f^{jj''j'}_{ii''i'}\] 
hence
\begin{equation}\label{eqn2.4} (F^{jj'}_{ii'}+G^{jj'}_{ii'})_{w^{j''}_{i''}}=(F^{jj''}_{ii''}+G^{jj''}_{ii''})_{w^{j'}_{i'}}\quad (i,i'=1,\ldots,n;\,j=1,\ldots,m), \end{equation} where $F^{jj'}_{ii'}$ are given but $G^{jj'}_{ii'}$ unknown functions.
The first requirement (\ref{eqn2.1}) is regarded as clarified for this moment.

Turning to the \emph{second requirement} (\ref{eqn2.1}), it implies
\begin{equation}\label{eqn2.5} (F^j)_{w^{j'}_{i'}}=f^{j'}_{i'w^j}-\sum f^{jj'}_{ii'x_i}-\sum f^{jj'}_{ii'w^{j''}}w^{j''}_i-f^j_{i'w^{j'}}.
\end{equation}
Denoting
\begin{equation}\label{eqn2.6} F^{j'j}_{i'}=(F^j)_{w^{j'}_{i'}}+\sum f^{jj'}_{ii'x_i}+\sum f^{jj'}_{ii'w^{j''}}w^{j''}_i, \end{equation}
there are obvious identities
\begin{equation}\label{eqn2.7} F^{j'j}_{i'}+F^{jj'}_{i'}=0,\  (F^{j'j}_{i'})_{w^{j''}_{i''}}= f^{j'j''}_{i'i''w^j}-f^{jj''}_{i'i''w^{j'}}.
 \end{equation}
Due to (\ref{eqn2.2}), every function (\ref{eqn2.6}) can be expressed in terms of functions
\begin{equation}\label{eqn2.8} F^j,\ F^{jj'}_{ii'},\ G^{jj'}_{ii'}\quad (i,i'=1,\ldots,n;\, j,j'=1,\ldots,m). \end{equation}
Identities (\ref{eqn2.7}) can be expressed in terms of functions  (\ref{eqn2.8}) as well and therefore may be regarded as necessary solvability conditions for the second requirement (\ref{eqn2.1}). Alas, they are not sufficient.

Indeed, assume the second identity (\ref{eqn2.7}). Then
\begin{equation}\label{eqn2.9} F^{j'j}_{i'}=f^{j'}_{i'w^j}-f^j_{i'w^{j'}}+B^{j'j}_{i'} \end{equation}
where \[B^{j'j}_{i'}=B^{j'j}_{i'}(\cdot\cdot,x_i,w^{j''},\cdot\cdot)\quad ({i'}=1,\ldots,n;\, j,j'=1,\ldots,m)\] are appropriate functions. With this result, the definition equation (\ref{eqn2.6}) reads
\[(F^j)_{w^{j'}_{i'}}=F^{j'j}_{i'}-\sum f^{jj'}_{ii'x_i}-\sum f^{jj'}_{ii'w^{j''}}w^{j''}_i=\frac{\partial}{\partial w^{j'}_{i'}}E^j[f]+B^{j'j}_{i'}\]
(direct verification) and therefore
\begin{equation}\label{eqn2.10} F^j=E^j[f]+\sum B^{j'j}_{i'}(w^{j'}_{i'}-c^{j'}_{i'})+B^j \end{equation}
where \[B^j=B^j(\cdot\cdot,x_i,w^{j'},\cdot\cdot)\quad (j=1,\ldots,m)\] are appropriate functions. The functions
\begin{equation}\label{eqn2.11} c^j_i=c^j_i(\cdot\cdot,x_{i'},w^{j'},\cdot\cdot)\qquad (i=1,\ldots,n;\, j=1,\ldots,m) \end{equation}
are fixed and may be arbitrarily chosen in advance. The second requirement (\ref{eqn2.1}) is satisfied if and only if $B^{j'j}_{i'}=B^j=0$ identically and this goal can be achieved as follows. 

Identity (\ref{eqn2.2}) may be regarded as a~system of differential equations
\begin{equation}\label{eqn2.12} \frac{\partial^2 f}{\partial w^j_i\partial w^{j'}_{i'}}=F^{jj'}_{ii'}+G^{jj'}_{ii'}\quad (i,i'=1,\ldots,n;\, j,j'=1,\ldots,m\,). \end{equation}
If $\bar f=\bar f(\cdot\cdot,x_i,w^j,w^j_i,\cdot\cdot)$ is the (unique) particular solution such that
\begin{equation}\label{eqn2.13} \bar f=\frac{\partial \bar f}{\partial w^j_i}=0\ (i=1,\ldots,n;\,j=1,\ldots,m)\quad \text{if}\quad w^1_1=c^1_1,\ldots, w^m_n=c^m_n \end{equation}
then the general solution is
\begin{equation}\label{eqn2.14} f=\bar f+\sum A^j_i(w^j_i-c^j_i)+A \end{equation}
where \[A^j_i=A^j_i(\cdot\cdot,x_{i'},w^{j'},\cdot\cdot), A=A(\cdot\cdot,x_{i'},w^{j'},\cdot\cdot)\quad (i=1,\ldots,n;\,j=1,\ldots,m)\] are arbitrary functions.
One can then see that identity (\ref{eqn2.9}) with $B^{jj'}_{i'}=0$ is equivalent to the equation
\[F^{j'j}_{i'}=(A^{j'}_{i'})_{w^j}-(A^j_{i'})_{w^{j'}}\quad \text{if}\quad w^1_1=c^1_1,\ldots, w^m_n=c^m_n\] for the coefficients $A^j_i.$
With the latent use of the first identity (\ref{eqn2.7}), this is expressed by the equation
\begin{equation}\label{eqn2.15} \sum F^{j'j}_{i'}dw^j\wedge dw^{j'}=2d_0\sum A^{j'}_{i'}dw^{j'}\quad \text{if}\quad w^1_1=c^1_1,\ldots, w^m_n=c^m_n \end{equation}
where differential $d_0$ is applied to variables $w^1,\ldots,w^m.$ Due to the Poincar\'e Lemma, we have necessary and sufficient condition
\begin{equation}\label{eqn2.16} d_0\sum F^{j'j}_{i'}dw^j\wedge dw^{j'}=0\ (i'=1,\ldots,n)\quad \text{if}\quad w^1_1=c^1_1,\ldots, w^m_n=c^m_n \end{equation}
for the existence of functions $A^j_i.$ Quite analogously, assuming already $B^{j'j}_{i'}=0,$ equation (\ref{eqn2.10}) with $B^j=0$ is ensured if and only if
\[F^j=A_{w^j}-\sum (A^j_i)_{x_i}-\sum (A^j_i)_{w^{j''}}w^{j''}_i\quad \text{if}\quad w^1_1=c^1_1,\ldots, w^m_n=c^m_n. \]
This is equivalent to the identity
\begin{equation}\label{eqn2.17} \sum\left\{F^j+\sum(A^j_i)_{x_i}+\sum(A^j_i)_{w^{j''}}c^{j''}_i\right\}dw^j=d_0A
\end{equation}
and we have necessary and sufficient condition
\begin{equation}\label{eqn2.18} d_0\sum\left\{\cdots\right\}dw^j=0\quad \text{if}\quad w^1_1=c^1_1,\ldots, w^m_n=c^m_n\end{equation}
for the existence of function $A.$

If (in principle arbitrary) functions (\ref{eqn2.11}) depend only on variables $x_1,\ldots,x_n\,,$ condition (\ref{eqn2.18}) can be expressed without the use of coefficients $A^j_i.$ Indeed, we may substitute
\[\begin{array}{l}d_0\sum(A^j_i)_{x_i}dw^j=\dfrac{\partial}{\partial x_i}d_0\sum A^j_idw^j,\\ d_0\sum(A^j_i)_{w^{j''}}c^{j''}_idw^j=c^{j''}_i\dfrac{\partial}{\partial w^{j''}}d_0\sum A^j_idw^j\end{array}\]
into (\ref{eqn2.18}) together with the use of (\ref{eqn2.15}) to obtain the concluding condition
\begin{equation}\label{eqn2.19}2d_0\sum F^jdw^j+\sum\left(\frac{\partial}{\partial x_i}+\sum c^{j''}_i\frac{\partial}{\partial w^{j''}}\right)\sum F^{j'j}_{i}dw^j\wedge dw^{j'}=0\end{equation}
which is equivalent to (\ref{eqn2.18}).

\medskip

{\bf Summary}: The first--order Lagrange function $f$ resolving the exact inverse problem (\ref{eqn2.1}) is given by (\ref{eqn2.14}) where $\bar f$ is a~particular solution of equation (\ref{eqn2.12}) satisfying the initial conditions (\ref{eqn2.13}) and coefficients $A^j_i, A$ satisfy (\ref{eqn2.15}) and (\ref{eqn2.17}). The necessary and sufficient solvability conditions for the existence of function $f$ are (\ref{eqn2.3}), (\ref{eqn2.4}), (\ref{eqn2.7}), (\ref{eqn2.16}) and (\ref{eqn2.18}) or (\ref{eqn2.19}). They are expressed in terms of functions (\ref{eqn2.8}) where $F^j$ and $F^{jj'}_{ii'}$ are given but $G^{jj'}_{ii'}$ are unknown.

\bigskip

This achievement looks rather involved than to be of any practical use, we nevertheless turn to particular examples in order to obtain more reasonable results. It is to be noted that analogous applications of the seemingly easier Helmholz conditions and the Tonti formula for the case $n>1$ of several independent variables are not available in actual literature.
\begin{remark}\label{rem1} \emph{The main idea of our strategy was as follows. The \emph{first requirement} (\ref{eqn2.1}) determines only the derivatives $f^{jj'}_{ii'}$ and we need the linear correction of $f$ to fulfil the \emph{second requirement} (\ref{eqn2.1}). Identities (\ref{eqn2.7}) appear after two differentiations. The original second requirement can be restored after two integrations  (\ref{eqn2.15}) and (\ref{eqn2.17}). Then the constants of integration determine just the linear correction of $f$ and we are done.}
\end{remark}

%%%%%%%%%%%%%%%%%%%%%%%%%%%%%%%%%%%%%%%%%%%%%%%%
\section{The case of one dependent variable}\label{sec3}
We suppose $m=1.$ Let us abbreviate
\[w_I=w^1_I, f_i=f^1_i=f_{w^1_i}, f_{ii'}=f^{11}_{ii'}=f_{w^1_iw^1_{i'}}, \ldots\]
but the remaining notation is retained. Then
\[f=f(\cdot\cdot,x_i,w,w_i,\cdot\cdot),\ e[f]=e^1[f]w_t,\ e^1[f]=E^1[f]-\sum E^{11}_{ii'}[f]w_{ii'}\]
where
\[ E^{11}_{ii'}=f_{ii'},\ E^1[f]=f_w-\sum f_{ix_i}-\sum f_{iw}w_i.\]
The exact inverse problem
\begin{equation}\label{eqn3.1}  F^{11}_{ii'}= E^{11}_{ii'}[f]\quad (i,i'=1,\ldots,n),\ F^1=E^1[f] \end{equation}
simplifies since $ G^{11}_{ii'}=0.$ The \emph{first requirement} (\ref{eqn3.1}) is solvable if and only if
\begin{equation}\label{eqn3.2} (F^{11}_{ii'})_{w_{i''}}=(F^{11}_{ii''})_{w_{i'}}\qquad (i,i',i''=1,\ldots,n) \end{equation}
where the symmetry $F^{11}_{ii'}=F^{11}_{i'i}$ is supposed. Passing to the \emph{second requirement} (\ref{eqn3.1}), one can infer that the vanishing $F^{11}_{i'}=0$ follows from the first identity (\ref{eqn2.7}) and then the definition (\ref{eqn2.6}) turns into the solvability condition
\begin{equation}\label{eqn3.3} (F^1)_{w_{i'}}+\sum(F^{11}_{ii'})_{x_i}+\sum(F^{11}_{ii'})_ww_i=0\qquad (i'=1,\ldots,n). \end{equation}
The second identity (\ref{eqn2.7}) becomes trivial since $j=j'=j''=1.$ Let us turn to the existence of coefficients $A^1_i$ and $A.$ First of all, (\ref{eqn2.15}) is simplified to $d_0\sum A^1_i dw=0$ whence $A^1_i=A^1_i(x_1,\ldots,x_n,w)$ may be arbitrary functions. Therefore only one additional equation
\begin{equation}\label{eqn3.4} A_w=F^1+\sum(A^1_i)_{x_i}+\sum(A^1_i)_ww_i\quad \text{ if }\quad w_1=c^1_1,\ldots,w_n=c^1_n \end{equation}
equivalent to (\ref{eqn2.17}) is nontrivial. We conclude:
\begin{theorem}\label{th3.1} The exact inverse problem $(\ref{eqn3.1})$ admits a~first--order solution $f$ if and only if conditions $(\ref{eqn3.2})$ and $(\ref{eqn3.3})$ are satisfied. Then
\begin{equation}\label{eqn3.5} f=\bar f+\sum A^1_i(w_i-c^1_i)+A\end{equation}
where $\bar f=\bar f(\cdot\cdot,x_i,w,w_i,\cdot\cdot)$ is the unique solution of the initial problem
\[\frac{\partial^2\bar f}{\partial w_i\partial w_{i'}}=F^{11}_{ii'}, \bar f=\frac{\partial\bar f}{\partial w_i}=0\quad \text{if}\quad w_1=c^1_1,\ldots,w_n=c^1_n\quad (i,i'=1,\ldots,n).\]
Functions $c^1_1,\ldots,c^1_n$ of variables $x_1,\ldots,x_n,w$ are arbitrary (but fixed) and $A^1_i,A$ may be arbitrary functions of the same variables satisfying $(\ref{eqn3.4}).$
\end{theorem}
A~complementary result is as follows.
\begin{theorem}\label{th3.2} If the exact inverse problem $(\ref{eqn3.1})$ admits a~solution $f,$ then it also admits the first--order solution. Alternatively saying, the Helmholz conditions imply $(\ref{eqn3.2})$ and $(\ref{eqn3.3})$. \end{theorem}
\begin{proof} Direct analysis of the identity $(\ref{eqn1.7}),$ that is, of the identity
\[e^1\left[(F^1-\sum F^{11}_{ii'}w_{ii'})w_t\right]=0,\]
in the extended jet space yields the condition
\[-\sum\frac{d}{dx_i}C_{ii''}w_t+2\sum C_{ii''}w_{it}=0\qquad (i''=1,\ldots,n)\]
where
\[C_{ii''}=-(F^1)_{w_{i''}}+\sum ( F^{11}_{ii'})_{w_{i''}}w_{ii'}-\sum\frac{d}{dx_{i'}} F^{11}_{ii'}.\]
Then the insertion of
\[\frac{d}{dx_{i'}} F^{11}_{ii'}=( F^{11}_{ii'})_{x_{i'}}+( F^{11}_{ii'})_ww_{i'}+\sum( F^{11}_{ii'})_{w_i}w_{ii'}\]
into the Helmholz conditions $C_{ii''}=0$ provides the desired result. \end{proof}
\begin{remark}\label{rem2} \emph{The explicit calculation of the exact inverse problem $(\ref{eqn3.1})$ is reduced to the equations $(\ref{eqn3.4})$ and $(\ref{eqn3.5})$. Therefore the investigation of global or multivalued solutions under the most general assumptions concerning the definition domains of functions $F^{11}_{ii'}$ and $F^1$ does not cause any difficulties since the well--known classical theory may be comfortably applied.} \end{remark}
%%%%%%%%%%%%%%%%%%%%%%%%%%%%%%%%%%%%%%%%%%%%%%%%
\section{The case of one independent variable}\label{sec4}
We suppose $n=1.$ Let us abbreviate
\[x=x_1, w^j_2=w^j_{11}, w^j_3=w^j_{111}, \ldots, f^j=f^j_1=f_{w^j_1}, f^{jj'}=f^{jj'}_{11}=f_{w^j_1w^{j'}_1}, \ldots \]
but otherwise the notation is retained. Then
\[f=f(x,\cdot\cdot,w^j,w^j_1,\cdot\cdot), e[f]=\sum e^j[f]w^j_t, e^j[f]=E^j[f]-\sum E^{jj'}_{11}[f]w^{j'}_2\]
where \[E^{jj'}_{11}[f]=f^{jj'},\ E^j[f]=f_{w^j}-(f^j)_x-\sum (f^j)_{w^{j'}}w^{j'}_1.\]
The exact inverse problem
\begin{equation}\label{eqn4.1} F^{jj'}_{11}=E^{jj'}_{11}[f],\ F^j=E^j[f]\qquad (j,j'=1,\ldots,m) \end{equation}
simplifies since $G^{jj'}_{11}=0.$ The \emph{first requirement} (\ref{eqn4.1}) is solvable if and only if
\begin{equation}\label{eqn4.2}(F^{jj'}_{11})_{w^{j''}_1}=(F^{jj''}_{11})_{w^{j'}_1}\qquad (j,j',j''=1,\ldots,m)\end{equation}
where the symmetry $F^{jj'}_{11}=F^{j'j}_{11}$ is supposed. Passing to the \emph{second requirement} (\ref{eqn4.1}), we recall the functions \[F^{j'j}_1=(F^j)_{w^{j'}_1}+(F^{jj'}_{11})_{x}+\sum(F^{jj'}_{11})_{w^{j''}}w^{j''}_1\]
and identities (\ref{eqn2.7}) which read
\begin{equation}\label{eqn4.3}F^{j'j}_1+F^{jj'}_1=0, (F^{j'j}_1)_{w^{j''}_1}=(F^{j'j''}_{11})_{w^j}-(F^{jj''}_{11})_{w^{j'}}\quad (j,j',j''=1,\ldots,m).
\end{equation}
There are additional solvability conditions
\begin{equation}\label{eqn4.4}d_0\sum F^{j'j}_1dw^j\wedge dw^{j'}=0,\end{equation}
\begin{equation}\label{eqn4.5}2d_0\sum F^jdw^j+\left(\frac{\partial}{\partial x}+\sum c^{j''}_1\frac{\partial}{\partial w^{j''}_1}\right)\sum F^{j'j}_1dw^j\wedge dw^{j'}=0\end{equation}
at the level set $w^1_1=c^1_1,\ldots,w^m_1=c^m_1$ where $c^j_1=c^j_1(x)$ are arbitrary but fixed functions. We conclude:
\begin{theorem}\label{th4.1}
The exact inverse problem $(\ref{eqn4.1})$ admits the first--order solution $f$ if and only if the identities $(\ref{eqn4.2})$--$(\ref{eqn4.5})$ are satisfied. Then the sought Lagrange function is
\[f=\bar f+\sum A^j_1(w^j_1-c^j_1)+A\] where $\bar f=\bar f(x,\cdot\cdot,w^j,w^j_1,\cdot\cdot)$ is the unique solution of the initial problem
\[\frac{\partial^2\bar f}{\partial w^j_1\partial w^{j'}_1}=F^{jj'}_{11}\ (j,j'=1,\ldots,m), \bar f=\frac{\partial\bar f}{\partial w^j_1}=0\quad \text{if}\quad w^1_1=c^1_1,\ldots,w^m_1=c^m_1\]
and coefficients $A^j_1,A$ satisfy
\[d_0\sum A^{j'}_1dw^{j'}=\sum F^{j'j}_1dw^j\wedge dw^{j'},\]
\[d_0A=\sum\left(F^j+(\frac{\partial}{\partial x}+\sum c^{j'}_1\frac{\partial}{\partial w^{j'}})A^j_1\right)dw^j\] at the level set $w^1_1=c^1_1,\ldots,w^m_1=c^m_1.$ Functions $c^1_1, \ldots,c^m_1$ of variable $x$ are arbitrary (but fixed). \end{theorem}
The complementary result is strong.
\begin{theorem}\label{th4.2}
If an~exact inverse problem with the second--order data
\[\label{eqn4.10}e^j=e^j[f]\qquad (j=1,\ldots,m;\,e^j=e^j(x,\cdot\cdot,w^{j'},w^{j'}_1,w^{j'}_2,\cdot\cdot)) \]
admits a~solution (of any order), then the given functions $e^j$ are in fact linear in the second--order variables and the exact inverse problem admits even the first--order solution. $($So we deal with the problem $(\ref{eqn4.1}).$\/$)$ \end{theorem}
\begin{proof}[A note to proof] Identity (\ref{eqn1.7}) in the extended jet space reads
\[\sum(e^{j'})_{w^j}w^{j'}_t-\frac{d}{dx}\left(\sum(e^{j'})_{w^j_1}w^{j'}_t\right)-\frac{d}{dt}e^j+\frac{d^2}{dx^2}\left(\sum(e^{j'})_{w^j_2}w^j_t\right)=0.\]
The linearity in $w^j_2$ immediately follows, however, our solvability conditions cannot be easily derived from the Helmholz conditions by a~mere formal calculus.
The complete proof rests on the \emph{reduction principle} \cite[Theorem 4.5.5]{T2}: \emph{If $n=1$ then every Euler--Lagrange system of even order $2K$ corresponds to appropriate Lagrange function of order $K.$} In our case $K=1.$ See the Appendix below. \end{proof}
\begin{remark}\label{rem3} \emph{The case $n=1$ is extensively discussed in \cite{T1, T2} with the systematical use of the Poincar\'e--Cartan forms and the final formulae are expressed by the Tonti integrals. The results  \cite{T1, T2} are of a~local nature: the multiple and global solutions cannot be included, the singularities are not allowed.
We also refer to a~familiar global theory \cite{T3} employing homology of total variational bicomplex which excludes the presence of the higher--order singularities in the fibers of the jet space.} \end{remark}
%%%%%%%%%%%%%%%%%%%%%%%%%%%%%%%%%%%%%%%%%%%%%%%%
\section{The case of two dependent variables}\label{sec5}
We suppose $m=2$ with the range of indices $i,i',i''=1,\ldots,n$ and $j,j',j''=1,2.$ The original notation is preserved. The given functions
\begin{equation}\label{eqn5.1} F^{jj'}_{ii'}=F^{jj'}_{i'i}=F^{j'j}_{i'i}=F^{j'j}_{ii'},\ F^j \end{equation}
are of the symmetrical nature while the auxiliary functions
\begin{equation}\label{eqn5.2} G^{jj}_{ii'}=G^{jj'}_{ii}=0,\ G^{12}_{ii'}=-G^{12}_{i'i}=G^{21}_{i'i}=-G^{21}_{ii'} \end{equation}
are of the skew--symmetric kind and need not identically vanish.  This fact makes the exact inverse problem nontrivial.

\medskip

Let us recall the main achievements of Section~\ref{sec2}.

\medskip

The \emph{first requirement} (\ref{eqn2.1}) was clarified by equations (\ref{eqn2.4}) which read
\begin{equation}\label{eqn5.3} 
\left(F^{jj}_{ii'}\right)_{w^j_{i''}}=\left(F^{jj}_{ii''}\right)_{w^j_{i'}}, \ \left(F^{12}_{ii}\right)_{w^1_{i'}}=\left(F^{11}_{ii'}\right)_{w^2_i}, \ \left(F^{21}_{ii}\right)_{w^2_{i'}}=\left(F^{22}_{ii'}\right)_{w^1_i},
\end{equation}
\begin{equation}\label{eqn5.4} 
\left(F^{12}_{ii'}+G^{12}_{ii'}\right)_{w^1_{i''}}=\left(F^{11}_{ii''}\right)_{w^2_{i'}},\ \left(F^{21}_{ii'}+G^{21}_{ii'}\right)_{w^2_{i''}}=\left(F^{22}_{ii''}\right)_{w^1_{i'}}
\end{equation} in our case $m=2.$ Identities (\ref{eqn5.3}) concern only the given functions (\ref{eqn5.1}) while identities (\ref{eqn5.4}) may be regarded as differential equations for the functions (\ref{eqn5.2})

The \emph{second requirement} (\ref{eqn2.1}) was represented by identities (\ref{eqn2.7}) completed with the solvability conditions (\ref{eqn2.16}) and (\ref{eqn2.19}). The first identity (\ref{eqn2.7}) reads
\begin{equation}\label{eqn5.5} 
\left(F^j\right)_{w^{j'}_{i'}}+\left(F^{j'}\right)_{w^j_{i'}}+2\sum\left(F^{jj'}_{ii'}\right)_{x_i}+2\sum\left(F^{jj'}_{ii'}\right)_{w^{j''}}w^{j''}_i=0,
\end{equation} by using definition (\ref{eqn2.6}) and the skew--symmetry (\ref{eqn5.2}). It concerns only the functions (\ref{eqn5.1}).
The second identity (\ref{eqn2.7}) is trivial if $j=j'.$ Assuming $j\not= j',$ we obtain equations
\begin{equation}\label{eqn5.6} \begin{array}{l}
\left(F^{12}_{i'}\right)_{w^1_i}=\left(F^{11}_{i'i}\right)_{w^2}-\left(F^{21}_{i'i}+G^{21}_{i'i}\right)_{w^1},\\ \left(F^{12}_{i'}\right)_{w^2_i}=\left(F^{12}_{i'i}+G^{12}_{i'i}\right)_{w^2}+\left(F^{22}_{i'i}\right)_{w^1}\end{array}
\end{equation}
which simplify if $i=i'.$
Condition (\ref{eqn2.16}) is trivial and condition (\ref{eqn2.19}) reads
\begin{equation}\label{eqn5.7}\left(F^1\right)_{w^2}-\left(F^2\right)_{w^1}+\sum\left(\frac{\partial}{\partial{x_{i}}}+\sum c^j_{i}\frac{\partial}{\partial w^j}\right)F^{12}_{i}=0\end{equation} if $w^1_1=c^1_1,\ldots, w^2_n=c^2_n.$ Functions $F^{12}_i$ appearing here are defined in (\ref{eqn2.6}).

\bigskip

In more detail. On this occasion, the lower indices $i,i',i''$ are completed with additional $k,k'=1,2$ for aesthetic reasons.
%%%%%%%%%%%%%%%%%%%%%%%
\begin{lemma}\label{lemma5.1} The skew--symmetry $(\ref{eqn5.2})$ is ensured if and only if identities
\begin{equation}\label{eqn5.8}
\left(F^{11}_{ik}\right)_{w^2_{i'}}+\left(F^{11}_{i'k}\right)_{w^2_i}=2\left(F^{12}_{ii'}\right)_{w^1_k},\ \left(F^{22}_{ik}\right)_{w^1_{i'}}+\left(F^{22}_{i'k}\right)_{w^1_i}=2\left(F^{12}_{ii'}\right)_{w^2_k}
\end{equation}
are satisfied. \end{lemma}
\begin{proof} We recall that (\ref{eqn5.4}) is regarded as the system of differential equations
\begin{equation}\label{eqn5.9} \frac{\partial G^{12}_{ii'}}{\partial w^1_k}=\left(F^{11}_{ik}\right)_{w^2_{i'}}-\left(F^{12}_{ii'}\right)_{w^1_{k}},\ \frac{\partial G^{21}_{ii'}}{\partial w^2_k}=\left(F^{22}_{ik}\right)_{w^1_{i'}}-\left(F^{12}_{ii'}\right)_{w^2_{k}}.
\end{equation}
The skew--symmetry (\ref{eqn5.2}) holds true if and only if the system
\[\frac{\partial G^{12}_{ii'}}{\partial w^1_k}=-\left(F^{11}_{i'k}\right)_{w^2_i}+\left(F^{12}_{i'i}\right)_{w^1_k},\ 
\frac{\partial G^{21}_{ii'}}{\partial w^2_k}=-\left(F^{22}_{i'k}\right)_{w^1_i}-\left(F^{21}_{i'i}\right)_{w^2_k}\]
appearing by the exchange $i\leftrightarrow i'$ is identical with (\ref{eqn5.9}). This observation immediately implies the assertion of Lemma~\ref{lemma5.1}. \end{proof}
\begin{lemma}\label{lemma5.2}
Assuming the skew--symmetry, system $(\ref{eqn5.9})$ is compatible if and only if all identities
\begin{equation}\label{eqn5.10}\begin{array}{l}
\left(F^{11}_{ik}\right)_{w^2_{i'}w^2_{k'}}=\left(F^{22}_{i'k'}\right)_{w^1_iw^2_k},\\ \left(F^{11}_{ik}\right)_{w^1_{i'}w^2_{k'}}=\left(F^{11}_{ii'}\right)_{w^2_kw^1_{k'}},\\ \left(F^{22}_{ik}\right)_{w^2_{i'}w^1_{k'}}=\left(F^{22}_{ii'}\right)_{w^1_kw^2_{k'}}
\end{array}\end{equation} are satisfied. \end{lemma}
\begin{proof}
Routine application of the rule
\[\left(G^{12}_{ii'}\right)_{w^1_kw^2_{k'}}=\left(G^{12}_{ii'}\right)_{w^2_{k'}w^1_k},\ \left(G^{12}_{ii'}\right)_{w^j_kw^j_{k'}}=\left(G^{12}_{ii'}\right)_{w^j_{k'}w^j_k}\]
to the equation (\ref{eqn5.9}) combined with the skew--symmetry. For instance, we have
\[(G^{12}_{ii'})_{w^1_kw^1_{k'}}=((F^{11}_{ik})_{w^2_{i'}}-(F^{12}_{ii'})_{w^1_k})_{w^1_{k'}}\, , 
(G^{12}_{ii'})_{w^1_{k'}w^1_k}=((F^{11}_{ik'})_{w^2_{i'}}-(F^{12}_{ii'})_{w^1_{k'}})_{w^1_k}\]
and therefore
\[\left(F^{11}_{ik}\right)_{w^2_{i'}w^1_{k'}}=\left(F^{11}_{ik'}\right)_{w^2_{i'}w^1_k}\]
which provides the second equation (\ref{eqn5.10}). \end{proof}
\begin{lemma}\label{lemma5.3} Equations $(\ref{eqn5.6})$ can be expressed only in terms of functions $(\ref{eqn5.1}).$ \end{lemma}
\begin{proof} Recalling the particular case
\[F^{12}_{i'}=(F^2)_{w^1_{i'}}+\sum (F^{21}_{ki'}+G^{21}_{ki'})_{x_k}+\sum (F^{21}_{ki'}+G^{21}_{ki'})_{w^{j''}}w^{j''}_k\]
of definition (\ref{eqn2.6}), let us deal with the first equation (\ref{eqn5.6}). We are interested only in the occurences of functions (\ref{eqn5.2}). So we have the equation
\[((F^{12}_{i'})_{w^1_i}=)\cdots+\sum(G^{21}_{ki'})_{x_kw^1_i}+\sum((G^{21}_{ki'})_{w^{j''}}w^{j''}_k)_{w^1_i}=\cdots -(G^{21}_{i'i})_{w^1}\,.\]
However
\[(G^{21}_{ki'})_{x_kw^1_i}=(G^{21}_{ki'})_{w^1_ix_k}=-(F^{11}_{ki})_{w^2_{i'}}+(F^{12}_{ii'})_{w^1_k},\]
\[
((G^{21}_{ki'})_{w^{j''}}w^{j''}_k)_{w^1_i}=(-(F^{11}_{ki})_{w^2_{i'}}+(F^{12}_{ii'})_{w^1_k})_{w^{j''}}+(G^{21}_{ii'})_{w^1}\]
with the use of the first equation (\ref{eqn5.9}). We are done since $G^{21}_{ii'}=-G^{21}_{i'i}.$ The second equation (\ref{eqn5.6})  is analogous. \end{proof}

The remaining solvability conditions (\ref{eqn5.7}) can be satisfied by the appropriate choice of functions $G^{12}_{ii'}\,.$ Indeed, the system (\ref{eqn5.9}) admits a~certain general solution
\begin{equation}\label{eqn5.11} G^{12}_{ii'}=\bar G^{12}_{ii'}+C_{ii'}\quad (C_{ii'}=C_{ii'}(x_1,\ldots,x_n,w^1,w^2),\,C_{ii'}=-C_{i'i})\end{equation}
where $\bar G^{12}_{ii'}$ is the unique particular solution of the same system (\ref{eqn5.9}) such that
\begin{equation}\label{eqn5.12}
\bar G^{12}_{ii'}=\frac{\partial \bar G^{12}_{ii'}}{\partial w^j_k}=0\quad\text{if}\quad w^1_1=c^1_1,\ldots, w^2_n=c^2_n.
\end{equation}
If (\ref{eqn5.11}) is inserted into condition (\ref{eqn5.7}), we obtain differential equation
\[\begin{array}{c}(F^1)_{w^2}-(F^2)_{w^1}=\\
=\sum(\frac{\partial}{\partial x_i}+\sum c^j_i\frac{\partial}{\partial w^j})((F^2)_{w^1_i}+\sum(F^{21}_{ki}-C_{ki})_{x_k}+\sum(F^{21}_{ki}-C_{ki})_{w^{j''}}w^{j''}_k)\\ \end{array} \]
for the functions $C_{ki}.$ It may be a~little simplified to the form
\begin{equation}\label{eqn5.13}
\left(F^1\right)_{w^2}-\left(F^2\right)_{w^1}=\end{equation}
\[=\sum(\frac{\partial}{\partial x_{i}}+\sum c^j_{i}\frac{\partial}{\partial w^j})(
(F^2)_{w^1_i}-(F^1)_{w^2_i}+2\sum(C_{ik})_{x_k}+2\sum(C_{ik})_{w^{j''}}w^{j''}_k)\]
if identity (\ref{eqn5.5}) is employed.

\bigskip

We can eventually summarize as follows.

\begin{theorem}[solvability]\label{th5.1} Identities  $(\ref{eqn5.3}),$  $(\ref{eqn5.5}),$  $(\ref{eqn5.6})$ with Lemma~$\ref{lemma5.3}$ applied,  together with $(\ref{eqn5.8})$ and  $(\ref{eqn5.10})$ provide necessary and sufficient solvability conditions for the given data  $(\ref{eqn5.1})$ of the exact inverse problem  $(\ref{eqn2.1})$ with $m=2.$ \end{theorem}
\begin{theorem}[auxiliary functions]\label{th5.2} Functions  $(\ref{eqn5.2})$ are (not uniquely) determined by formula  $(\ref{eqn5.11})$ where $\bar G^{12}_{ii'}$ is (a~unique) particular solution of system  $(\ref{eqn5.9})$ satisfying  $(\ref{eqn5.12}).$ The ``integration constants $C_{ii'}$'' satisfy the second--order differential equation  $(\ref{eqn5.13})$ at the level set $w^1_1=c^1_1,\ldots,w^2_n=c^2_n.$ Functions $c^j_i=c^j_i(x_1,\ldots,x_n)$ may be arbitrarily chosen in advance. \end{theorem}
\begin{theorem}[the solution]\label{th5.3} The solution $f$ of the exact inverse problem  $(\ref{eqn2.1})$ is given by formula  $(\ref{eqn2.14})$ where $\bar f$ is a~particular solution of system  $(\ref{eqn2.12})$ satisfying  $(\ref{eqn2.13})$ and the coefficients $A^j_i, A$ are given by (solvable) equations  $(\ref{eqn2.15})$ and  $(\ref{eqn2.17})$ for the particular case $m=2.$
\end{theorem}
%%%%%%%%%%%%%%%%%%%%%%%%%%%%%%%%%%%%%%%%%%%%%%%%
\section{Concluding remarks}\label{s6}
\subsection{Two independent variables} We suppose $n=2.$ Then the skew--symmetry
\[G^{jj}_{ii'}=G^{jj'}_{ii}=0,\ G^{jj'}_{12}=-G^{jj'}_{21}=G^{j'j}_{21}=-G^{j'j}_{12}\]
of the auxiliary functions differs from (\ref{eqn5.2}) by a~mere exchange of the role of the upper and the lower indices. Analogous rule holds for the identities (\ref{eqn5.3})--(\ref{eqn5.5}) and for the Lemma~\ref{lemma5.1} and  Lemma~\ref{lemma5.2}. For instance, the auxiliary functions satisfy the system
\[\frac{\partial G^{jj'}_{12}}{\partial w^{j''}_1}=(F^{jj''}_{11})_{w^{j'}_2}-(F^{jj'}_{12})_{w^{j''}_1},\ 
  \frac{\partial G^{jj'}_{12}}{\partial w^{j''}_2}=(F^{jj''}_{22})_{w^{j'}_1}-(F^{jj'}_{12})_{w^{j''}_2}\]
analogous to (\ref{eqn5.9}) with the compatibility conditions
\[(F^{jj''}_{11})_{w^{j'}_2}+(F^{j'j''}_{11})_{w^j_2}=2(F^{jj'}_{12})_{w^{j''}_1},\ (F^{jj''}_{22})_{w^{j'}_1}+(F^{j'j''}_{22})_{w^j_1}=2(F^{jj'}_{12})_{w^{j''}_2}  \]
analogous to (\ref{eqn5.10}). However, the functions $F^{j'j}_{i'}$ in equations (\ref{eqn5.6}) destroy this principle and Lemma~\ref{lemma5.3} is not true. In more detail, we obtain identities
\[\begin{array}{l}
(F^{j'j}_1)_{w^{j''}_2}=(F^{j'j''}_{12}+G^{j'j''}_{12})_{w^j}-(F^{jj''}_{12}+G^{jj''}_{12})_{w^{j'}},\\
(F^{j'j}_2)_{w^{j''}_1}=(F^{j'j''}_{12}-G^{j'j''}_{12})_{w^j}-(F^{jj''}_{12}-G^{jj''}_{12})_{w^{j'}} \end{array}\]
and they provide additional ``cyclic'' differential equations such that
\begin{equation}\label{eqn6.1}
(G^{jj'}_{12})_{w^{j''}}+(G^{j'j''}_{12})_{w^{j}}+(G^{j''j}_{12})_{w^{j'}}\end{equation}
are certain functions of given data $F^{jj'}_{ii'}$ and $F^j.$ If $m=3,$ there is only one equation (\ref{eqn6.1}) and the inverse problem can be resolved in a~reasonable space. In any case, the general strategy does not fail.
%%%%%%%%%%%
\subsection{The first--order problem}\label{ss6.2} Let us suppose $F^{jj'}_{ii'}=0$ identically. Then \[f^{jj'}_{ii'}=G^{jj'}_{ii'}=-G^{jj'}_{i'i}=-f^{jj'}_{i'i}\] and it follows that all derivatives $f^{j_1\cdots j_k}_{\, i_1\cdots i_k}$ are skew--symmetric both in the lower and in the upper indices. Therefore $f$ is a~polynomial in variables $w^j_i$ and even
\[f=\sum A^{j_1\cdots j_k}_{\, i_1\cdots i_k}(\cdot\cdot,x_i,w^j,\cdot\cdot)\det\left(
\begin{array}{ccc}
w^{j_1}_{i_1}&\ldots&w^{j_1}_{i_k}\\
\ldots&&\ldots\\
w^{j_k}_{i_1}&\ldots&w^{j_k}_{i_k}
\end{array}\right)\]
\[(\text{sum over }i_1<\cdots <i_k;\, j_1<\cdots <j_k)\]
is of a~very special kind. The Helmholz conditions for the functions $F^j$ appear as follows. We start with the identity $e[\sum F^jw^j_t]=0$ in the extended jet space. In more detail
\[e^{j'}\left[\sum F^jw^j_t\right]=\sum(F^j)_{w^{j'}}-\sum\frac{d}{dx_i}\left(\sum F^j_{w^{j'}_i}w^j_t\right)-\frac{d}{dt}F^{j'}=0,\] which provides the solvability conditions
\[\begin{array}{l}
(F^j)_{w^{j'}}-\sum (F^j)_{w^{j'}_ix_i}-\sum (F^j)_{w^{j}_{i'}w^{j''}}w^{j''}_i-(F^{j'})_{w^j}=0,\\
(F^j)_{w^{j'}_i}+(F^{j'})_{w^j_i}=0,\ (F^j)_{w^{j'}_iw^{j''}_{i'}}+(F^j)_{w^{j'}_{i'}w^{j''}_{i}}=0.
\end{array}\]
Analogously as above, it follows that
\[F^j=\sum B^{jj_1\cdots j_k}_{\ \,i_1\cdots i_k}(\cdot\cdot,x_i,w^j,\cdot\cdot)\det\left(
\begin{array}{ccc}
w^{j_1}_{i_1}&\ldots&w^{j_1}_{i_k}\\
\ldots&&\ldots\\
w^{j_k}_{i_1}&\ldots&w^{j_k}_{i_k}
\end{array}\right)\]
\[(\text{sum over }i_1<\cdots <i_k;\, j_1<\cdots <j_k)\]
are polynomials in the first--order variables $w^j_i.$ The Tonti integral
\[\tilde f=\int_0^1\sum t^kB^{jj_1\cdots j_k}_{\ \,i_1\cdots i_k}(\cdot\cdot,x_i,tw^j+(1-t)c^j,\cdot\cdot)dt\,(w^j-c^j)\det\left(
\begin{array}{ccc}
w^{j_1}_{i_1}&\ldots&w^{j_1}_{i_k}\\
\ldots&&\ldots\\
w^{j_k}_{i_1}&\ldots&w^{j_k}_{i_k}
\end{array}\right)\]
where the fixed constants $c^1,\ldots,c^m\in\mathbb R$ may be regarded for a~reasonable solution of this marginal inverse problem which is extensively discussed in \cite{T4}.
%%%%%%%%%%%
\subsection{The reducible case}\label{ss6.3}
In more generality, let us suppose
\[\begin{array}{l}
F^{jj'}_{ii'}=F^{jj'}_{ii'}(\cdot\cdot,x_{i''},w^{j''},\cdot\cdot)\quad\text{if}\quad j\neq j'\quad\text{and}\quad i\neq i'\,,\\
F^{jj}_{ii'}=F^{jj}_{ii'}(\cdot\cdot,x_{i''},w^{j''},w^j_i,w^j_{i'}\cdot\cdot),\ F^{jj'}_{ii}=F^{jj'}_{ii}(\cdot\cdot,x_{i''},w^{j''},w^j_i,w^{j'}_i\cdot\cdot)
\end{array}\]
where $i,i'=1,\ldots,n$ and $j,j'=1,\ldots,m.$ Then the solvability conditions (\ref{eqn2.4}) separately concern either only the given functions $F^{jj'}_{ii'}$ or the auxiliary functions $G^{jj'}_{ii'}.$ In more detail, if the identities
\[(F^{jj}_{ii'})_{w^j_i}=(F^{jj}_{ii})_{w^j_{i'}},\ (F^{jj'}_{ii})_{w^j_i}=(F^{jj}_{ii})_{w^{j'}_i}\]
are satisfied, then there exists function $\bar f=\bar f(\cdot\cdot,x_i,w^j,w^j_i,\cdot\cdot)$ such that
\[\frac{\partial^2\bar f}{\partial w^j_i\partial w^{j'}_{i'}}=F^{jj'}_{ii'}\quad (i,i'=1,\ldots,n;\, j,j'=1,\ldots,m).\]
If $f$ is a~solution of the exact inverse problem (\ref{eqn2.1}) then
\[\begin{array}{c}
E^{jj'}_{ii'}[f-\bar f]=E^{jj'}_{ii'}[f]-E^{jj'}_{ii'}[\bar f]=0,\\
E^j[f-\bar f]=E^j[f]-E^j[\bar f]= F^j-E^j[\bar f].\end{array}\]
It follows that $g=f-\bar f$ is the solution of the first--order inverse problem
\[0=E^{jj'}_{ii'}[g],\ F^j-E^j[\bar f]=e^j[g]\]
 of the preceding point~\ref{ss6.2}.
%%%%%%%%%%%
\subsection{Higher--order inverse problems} We believe that some reducible higher--order Euler--Lagrange expressions appearing in applications can be investigated analogously to the preceding point~\ref{ss6.3}. However the simple dichotomy of symmetry and skew--symmetry between the functions $F^{jj'}_{ii'}$ and $G^{jj'}_{ii'}$ turns into rather involved combinatorial structure for the general case of the higher--order inverse problem which deserves more place and another large article. %It seems that no progress should be expected at an early date.
%%%%%%%%%%%%%%%%%%%%%%%%%%%%%%%%%%%%%%%%%%%%%%%%
\section*{Appendix}
We suppose $n=1$ with the same alternative notation of the jet coordinates
\[x=x_1, w^j_r=w^j_{1\cdots 1}\qquad (r \text{ terms } 1\cdots 1;\, j=1,\ldots,m)\]
as in Section~\ref{sec4} above. Recalling the Lagrange functions and the Euler--Lagrange expressions
\[f=f(x,\cdot\cdot,w^j_r,\cdot\cdot),\quad e^j[f]=\sum(-1)^r\frac{d^r}{dx^r}\frac{\partial f}{\partial w^j_r}\quad (j=1,\ldots,m),\]
we can state short proof of the following result.\\\\
{\bf Proposition.} \emph{Let the Euler--Lagrange expressions $e^j[f]$ be of the order $S$ (at most). If $S=2K$ is even then $e^j[f]=e^j[g]$ for an~appropriate Lagrange function $g$ of the order $K.$ If $S=2K+1$ is odd then \[e^j[f]=e^j[g_0+\sum g_kw^k_{K+1}]\] where $g_0,\ldots,g_m$ are appropriate functions of the order $K.$} %\end{prop}
\begin{proof} Let \[ \varphi=fdx+\sum a^j_r\omega^j_r\qquad (\omega^j_r=dw^j_r-w^j_{r+1}dx)\]
be the familiar Poincar\'e--Cartan form of function $f.$ It is uniquely determined by the property
\[d\varphi\sim\sum e^j[f]\,\omega^j_0\wedge dx\qquad (\text{ mod all forms } \omega^{j'}_r\wedge\omega^{j''}_s).\]
If $f$ is of the order $R$ at most, the explicit formulae
\[a^j_r=0\ (r\geq R),\ a^j_{R-1}=\frac{\partial f}{\partial w^j_R},\ a^j_{r-1}=\frac{\partial f}{\partial w^j_{r}}-\frac{d}{dx}a^j_r\quad (r=R-1,\ldots,1)\]
for the coefficients $a^j_r$ are well--known.

 We are passing to the proof proper. In more detail, let us denote
\[d\varphi=\sum e^j[f]\,\omega^j_0\wedge dx+\sum_{r\geq s}a^{kj}_{rs}\,\omega^k_r\wedge\omega^j_s.\]
Then, modulo all forms $\omega^j_r\wedge\omega^{j'}_s\wedge\omega^{j''}_t$, we have 
\[0=d^2\varphi\sim\sum\frac{\partial e^j[f]}{\partial w^k_r}\omega^k_r\wedge\omega^j_0\wedge dx+\]
 \[+\left(\sum\frac{d}{dx}a^{kj}_{rs}\cdot\omega^k_r\wedge\omega^j_s+\sum a^{kj}_{rs}\cdot(\omega^k_{r+1}\wedge\omega^j_s+\omega^k_r\wedge\omega^j_{s+1})\right)\wedge dx\]
and it follows successively that
\[\begin{array}{l}
a^{kj}_{r0}=0\ (r\geq S), \text{ see the top--order factor of } \cdot\cdot\wedge\omega^j_0\wedge dx,\\
a^{kj}_{r1}=0\ (r\geq S-1), \text{ see the top--order factor of } \cdot\cdot\wedge\omega^j_1\wedge dx,\\
\cdots\\
a^{kj}_{r,K-1}=0\ (r\geq S-K+1), \text{ see the top--order factor of }\cdot\cdot\wedge\omega^j_{K-1}\wedge dx,\\
a^{kj}_{rK}=0\ (\text{ all } r \text{ if } S=2K;\,r>S-K \text{ if } S=2K+1),\\
a^{kj}_{r,K+1}=0\ (\text{ all } r \text{ if } S=2K+1),\\
a^{kj}_{rs}=0 \text{ if } r>K+1.
\end{array}\]
After this observation, the concluding part of the proof easily follows.

Assuming $S=2K,$ we have
$d\varphi\sim 0$ (mod $dx,$ all  $\omega^j_r$ with $r\leq K-1)$
hence
\[d\varphi\sim 0\quad (\text{mod}\ dx, \text{all } dw^j_r \text{ with } r\leq K-1)\]
and therefore
\[\varphi=dA+adx+\sum^{K-1} a^j_rdw^j_r=dA+\psi\]
by applying the Poincar\'e Lemma. However
\[\psi=adx+\sum a^j_rdw^j_r=gdx+\sum a^j_r\omega^j_r,\quad g=a+\sum a^j_rw^j_{r+1}\]
and the equality $d\psi=d\varphi$ implies that $\psi$ is the Poincar\'e--Cartan form for the Lagrange function $g$ of the order $K$ at most.

Assuming $S=2K+1,$ then analogously
\[\varphi=dA+\psi,\ \psi=gdx+\sum a^j_r\omega^j_r\quad (r\leq K)\]
and $\psi$ is the Poincar\'e--Cartan form of $g$ which is of the order $K+1$ at most. However we suppose that the Euler--Lagrange expressions of $g$ are of the order $S=2K+1$ and this implies the linearity in variables $w^j_{K+1}.$
\end{proof}

The proof rests on the same idea as in~\cite[p. 56--68]{T2}. Though it is much shorter, the use of the Poincar\'e--Cartan form obscures the elementary nature of the result since the direct approach is quite simple. Indeed, assume $n=1$ and let $f$ be a~Lagrange function just of the order $K.$ Then the Euler--Lagrange expressions are
\[\begin{array}{llr}
e^j[f]=\cdots+(-1)^K\sum f_{w^j_Kw^{j'}_K}w^{j'}_{2K}&(f=f(x,w^1,\ldots,w^m_K)),&(A)\\
e^j[f]=\cdots+(-1)^K\sum f^j_{w^{j'}_{K-1}}w^{j'}_{2K-1}&(f=\cdots+\sum f^j(x,w^1,\ldots,w^m_{K-1})w^j_K)&(B)
\end{array}\]
according to whether $f$ is nonlinear or linear in the top--order variables. We have
\[\begin{array}{ccr}
\text{max order }e^j[f]=2\,\text{order}\,f\text{ in } (A)&\\
&&\qquad(C)\\
\text{max order }e^j[f]=2\,\text{order}\,f-1\text{ in } (B)&
\end{array}\]
except for the case when $\partial f^j/\partial w^{j'}_{K-1}=0$ identically. However if \[f^j=f^j(x,w^1,\ldots,w^m_{K-2})\quad (j=1,\ldots,m)\] then $f$ in $(B)$ can be replaced with the Lagrange function
\[\bar f=f-\frac{d}{dx}\sum f^jw^j_{K-1}\quad \left(e^j[f]=e^j[\bar f];\, j=1,\ldots,m\right)\]
of the lower order. We conclude that there does exist the Lagrange function exactly satisfying both subcases $(C).$ This is just the Proposition.

\subsection*{Conflict of Interests}

The authors declare that there is no conflict of interests regarding the publication of this paper.

\subsection*{Acknowledgements}
This paper was elaborated with the financial support of the European
Union's "Operational Programme Research and Development for
Innovations", No. CZ.1.05/2.1.00/03.0097, as an activity of the
regional Centre AdMaS "Advanced Materials, Structures and
Technologies".

%%%%%%%%%%%%%%%%%%%%%%%%%%%%%%%%%%%%%%%%

*Brno University of Technology

Faculty of Civil Engineering

Department of Mathematics

Veve\v{r}\'{\i} 331/95, 602 00 Brno

Czech Republic

email: chrastinova.v@fce.vutbr.cz

\bigskip

**Brno University of Technology

Faculty of Civil Engineering

AdMaS Center

Veve\v{r}\'{\i} 331/95, 602 00 Brno

Czech Republic

email: tryhuk.v@fce.vutbr.cz

\end{document}